\documentclass[12pt]{article}
\usepackage{amsfonts,amsmath}
\usepackage{graphicx}
\usepackage{enumitem}
\usepackage{fullpage}

\title{Unexpected behaviour of crossing sequences}

\author{Matt DeVos \and
  Bojan Mohar\thanks{Supported in part by the
  Research Grant P1--0297 of ARRS (Slovenia), by an NSERC Discovery Grant (Canada)
  and by the Canada Research Chair program.}~\thanks{On leave from:
  IMFM \& FMF, Department of Mathematics, University of Ljubljana, Ljubljana,
  Slovenia.}
  \and
   Robert \v{S}\'{a}mal\thanks{Supported by PIMS postdoctoral fellowship.}
   \thanks{Partially supported by Institute for Theoretical Computer Science (ITI)}
   \thanks{On leave from Institute for Theoretical
     Computer Science (ITI), Charles University, Prague, Czech Republic.}
}

\newtheorem{theorem}{Theorem}[section]
\newtheorem{lemma}[theorem]{Lemma}

\newtheorem{proposition}[theorem]{Proposition}

\newtheorem{conjecture}[theorem]{Conjecture}
\newtheorem{problem}[theorem]{Problem}
\newenvironment{proof}{\par\medskip\noindent{\bf Proof:\ }}
  {\hskip 2cm\unskip\hbox{}\hfill$\Box$\par\bigskip}
\newenvironment{proofof}{\par\medskip\proofof}
  {\unskip\hfill$\Box$\par\bigskip}
\def\proofof #1{\noindent{\bf Proof of #1:\hskip 0.5em}}

\newcommand\s{{\mathcal S}}
\renewcommand\ss{{\mathbb S}}
\newcommand\calD{{\mathcal D}}
\def\Zet{{\mathbb Z}}
\def\En{{\mathbb N}}
\def\mytit#1#2{
\bigskip
{
  \leftskip=\parindent
  \noindent{\bf #1} #2
  \nopagebreak
  \par
}
\smallskip
}

\begin{document}

\date{}
\maketitle
\begin{center}
  {Department of Mathematics}\\
  {Simon Fraser University}\\
  {Burnaby, B.C. V5A 1S6} \\
  email: {\tt \{mdevos,mohar\}@sfu.ca, samal@kam.mff.cuni.cz}
\end{center}

\begin{abstract}
The $n^{th}$ crossing number of a graph $G$, denoted $cr_n(G)$, is
the minimum number of crossings in a drawing of $G$ on an orientable surface of
genus~$n$.  We prove that for every $a>b>0$, there exists a graph
$G$ for which $cr_0(G) = a$, $cr_1(G) = b$, and $cr_2(G) = 0$.
This provides support for a conjecture of Archdeacon et al.\ and
resolves a problem of Salazar.
\end{abstract}

\section{Introduction}

Planarity is ubiquitous in the world of structural graph theory, and
perhaps the two most obvious generalizations of this concept---crossing
number, and embeddings in more complicated surfaces---are topics which
have been thoroughly researched. Despite this, relatively little
work has been done on the common generalization of these two:
crossing numbers of graphs drawn on surfaces.  This subject seems to
have been introduced in \cite{Sir}, and studied further in \cite{ABS}.
Following these authors, we define for every nonnegative integer $i$ and every
graph $G$, the $i^{th}$ \emph{crossing number}, $cr_i(G)$, (and
also the $i^{th}$ \emph{nonorientable crossing number}, $\tilde{cr}_i(G)$)
to be the minimum number of crossings in a drawing of $G$ on the orientable
(nonorientable, respectively) surface of genus $i$. 
We consider drawings where each vertex $x$ of $G$ is represented by a point $\phi(x)$ 
of the suface, each edge $uv$ by a curve with ends at points $\phi(u)$ and $\phi(v)$ 
and with interior avoiding all points $\phi(x)$ for $x\in V(G)$. Moreover, we assume 
that no three edges are drawn so that they have an interior point in common. 
Observe that
$cr_i(G)=0$ (respectively, $\tilde{cr}_i(G)=0$) if and only if
$i$ is greater or equal to the genus (resp., nonorientable genus) of $G$.
This gives, for every graph $G$, two finite sequences of integers,
$(cr_0(G), cr_1(G),\ldots,0)$ and $(\tilde{cr}_0(G), \tilde{cr}_1(G),\ldots,0)$,
both of which terminate with a single zero.
The first of these is the {\it orientable crossing sequence} of $G$, the
second the {\it nonorientable crossing sequence} of~$G$.

A natural question is to characterize crossing sequences of graphs. This is the
focus of both \cite{Sir} and \cite{ABS}. If we are given a drawing of a graph in a
surface $\s$ with at least one crossing, then modifying our surface
in the neighborhood of this crossing by either adding a crosscap or a
handle gives rise to a drawing of $G$ in a higher genus surface with one
crossing less.  It follows from this that every orientable and
nonorientable crossing sequence is strictly decreasing until it
hits $0$. This necessary condition was conjectured to be sufficient
in~\cite{ABS}.

\begin{conjecture}[Archdeacon, Bonnington, and \v{S}ir\'a\v{n}]
\hfill\break If\/ $(a_1,a_2,\ldots,0)$ is a sequence of integers which
strictly decreases until $0$, then there is a graph whose crossing
sequence (nonorientable crossing sequence) is $(a_1,a_2,\ldots,0)$.
\end{conjecture}

To date, there has been very little progress on this appealing
conjecture. For the special case of sequences of the form $(a,b,0)$,
Archdeacon, Bonnington, and \v{S}ir\'a\v{n} \cite{ABS} constructed some
interesting examples for both the orientable and nonorientable
cases.  We shall postpone discussion of their examples for the
oriented case until later, but let us highlight their result for the
nonorientable case here.

\begin{theorem}[Archdeacon, Bonnington, and \v{S}ir\'a\v{n}]
\label{thm:ABS1}
If $a$ and\/ $b$ are integers with $a > b > 0$, then
there exists a graph $G$ with nonorientable crossing sequence
$(a,b,0)$.
\end{theorem}

It has been believed by some that such a result cannot hold for the
orientable case. For the most extreme special case $(N,N-1,0)$,
where $N$ is a large integer, Salazar asked \cite{SalazarConj} if
this sequence could really be the crossing sequence of a graph. The
following
quote of Dan Archdeacon illustrates why
such crossing sequences are counterintuitive:
\begin{quote}
If $G$ has crossing sequence $(N,N-1,0)$, then adding one handle enables
us to get rid of no more than a single crossing, but by adding the second
handle, we get rid of many. So, why would we not rather add the second
handle first?
\end{quote}

Our main theorem is an analogue of Theorem \ref{thm:ABS1} for the
orientable case, and its special case $a=N$, $b=N-1$ resolves
Salazar's question~\cite{SalazarConj}.

\begin{theorem} \label{hamburger}
If\/ $a$ and\/ $b$ are integers with $a > b > 0$, then there exists a graph $G$
whose orientable crossing sequence is $(a,b,0)$.
\end{theorem}

Quite little is known about constructions of graphs for more general
crossing sequences.  Next we shall discuss the only such construction we
know of.  Consider a sequence ${\bf a} = (a_0, a_1, \ldots,a_g)$ and define
the sequence $(d_1,\ldots,d_g)$ by the
rule $d_i = a_{i-1} - a_i$.  If ${\bf a}$ is the crossing sequence of a graph, then, roughly speaking, $d_i$ is the number of crossings which can be saved by adding the $i^{th}$ handle.
It seems intuitively clear that sequences for which
$d_1 \ge d_2 \ge \cdots \ge d_g$ should be crossing sequences, since here we
receive diminishing returns for each extra handle we use. Indeed,
\v{S}ir\'a\v{n} \cite{Sir} constructed a graph with crossing sequence
${\bf a}$ whenever $d_1 \ge d_2 \ge \cdots \ge d_g$.

Constructing graphs for sequences which violate the above condition
is rather more difficult.  For instance, it was previously open
whether there exist graphs with crossing sequence $(a,b,0)$ where
$a/b$ is arbitrarily close to~$1$.  The most extreme examples are
due to Archdeacon, Bonnington and \v{S}ir\'a\v{n} \cite{ABS} and
have $a/b$ approximately equal to $6/5$. Although our main theorem
gives us a graph with every possible crossing sequence of the form
$(a,b,0)$, we don't know what happens for longer sequences.  In
particular, it would be nice to resolve the following problem which
asks for graphs where the first $s$ handles save only an epsilon
fraction of what is saved by the $s+1^{st}$ handle.

\begin{problem} \label{op:longseq}
For every positive integer $s$ and every $\varepsilon > 0$, construct
a graph~$G$ for which
$cr_0(G) - cr_s(G) \le \varepsilon \left( cr_s(G) - cr_{s+1}(G) \right)$.
\end{problem}

For graph embeddings, the genus of a disconnected graph is the sum
of the genera of its connected components.  For drawing, this
situation is presently unclear.  If we have a graph which is a
disjoint union of $G_1$ and $G_2$, then we can always ``use part of
the surface for~$G_1$ and the other part for~$G_2$'', leading to
$$
  cr_i(G_1 \cup G_2) \le \min_j \bigl( cr_j(G_1) + cr_{i-j}(G_2) \bigr) \,.
$$
To the best of our knowledge, this inequality might always be an
equality.  More generally we shall pose the following problem.

\begin {problem}   \label{op:disjoint}
Let $G$ be a disjoint union of the graphs $G_1$ and $G_2$, and let
${\mathcal S}$ be a (possibly nonorientable) surface. Is there an
optimal drawing of\/ $G$ on ${\mathcal S}$, such that no edge of\/
$G_1$ crosses an edge of\/~$G_2$?
\end {problem}

This problem is trivially true when ${\mathcal S}$ is the plane, but
it also holds when ${\mathcal S}$ is the projective plane:  

\begin{proposition}\label{proj_plane}
Let $G$ be a disjoint union of the graphs $G_1$ and $G_2$. Then
$$
  \tilde{cr}_1(G) = \min \{ \tilde{cr}_1(G_1) + {cr}_0(G_2), {cr}_0(G_1) + \tilde{cr}_1(G_2) \} \,.
$$
In other words, there is an optimal drawing of $G$ where planar drawing
of $G_2$ is put into one of the regions defined by the drawing of $G_1$;
or vice versa.
\end{proposition}

\begin{proof}
To see this, consider an optimal drawing of $G$ on the 
projective plane, and suppose (for a contradiction) that some edge
of $G_1$ crosses an edge of $G_2$.  If there is a crossing involving
two edges in $G_1$, then by creating a new vertex at this crossing
point, we obtain an optimal drawing of this new graph. Continuing in
this manner, we may assume that both $G_1$ and $G_2$ are
individually embedded in the projective plane.  For $i=1,2$, let
$a_i$ be the length of a shortest noncontractible cycle in the dual
graph of the embedding of $G_i$.  Note that $a_i \ge 2$ as otherwise
$G_i$ embeds in the plane, so $G$ embeds in the projective plane.
Assume (without loss) that $a_1 \le a_2$.  Now, it follows from a
theorem of Lins \cite{Lins} that there exists a half-integral packing
of noncontractible cycles in $G_i$ with total weight $a_i$ for
$i=1,2$. Since any two noncontractible curves in the projective
plane meet, it follows that the total number of crossings in this
drawing is at least $a_1 a_2$. However, we can draw $G$ in the
projective plane by embedding $G_2$ and then drawing $G_1$ in a face
of this embedding with a total of ${a_1 \choose 2} =
\frac{1}{2}a_1(a_1 - 1) < a_1 a_2$ crossings, a contradiction.
\end{proof}

Our primary family of graphs used in proving Theorem~\ref{hamburger}
can be constructed with relatively
little machinery, so we shall introduce them here. We will however use a couple
of gadgets which are common in the study of crossing numbers (\cite{ABS,Pels}). 
Let us pause here to define
them precisely. A \emph{special graph} is a graph $G$ together with a
distinguished subset $T \subseteq E(G)$ of \emph{thick} edges, a
subset $U \subseteq V(G)$ of \emph{rigid} vertices and a family
$\{ \pi_u \}_{u \in U}$ of prescribed \emph{local rotations}
for the rigid vertices. Here, $\pi_u$ describes the cyclic ordering of
the ends of edges incident with $u$.
A \emph{drawing} of a special graph $G$ in a surface $\Sigma$ is a drawing
of the underlying graph $G$ with the added property that for
every $u \in U$, the local rotation of the edges incident with $u$ given
by this drawing either in the local clockwise or counterclockwise
order matches $\pi_u$.  The \emph{crossing number} of a
drawing of the special graph $G$ is $\infty$ if there is an edge in
$T$ which contains a crossing, and otherwise it is the same as the
crossing number of the drawing of the underlying graph.  We define the
\emph{crossing number} of a special graph $G$ in a surface $\Sigma$ to be the
minimum crossing number of a drawing of $G$ in $\Sigma$, and
$cr_i(G)$ to be the crossing number of $G$ in a surface of genus
$i$.  In the next section, we shall prove the following result.

\begin{lemma}\label{gadget_lem}
If $G$ is a special graph with crossing sequence ${\bf a}$ consisting of real numbers, then
there exists an (ordinary) simple graph with crossing sequence ${\bf a}$.
\end{lemma}

This result permits us to use special graphs in our constructions.
Indeed, starting in the third section, we shall consider special
graphs on par with ordinary ones, and we shall drop the term
special.  When defining a (special) graph with a diagram, we shall
use the convention that thick edges are drawn thicker, and vertices
which are marked with a box instead of a circle have the
distinguished rotation scheme as given by the figure. With this
terminology, we can now introduce our principal family of graphs.

\begin {figure}
\centerline{\includegraphics[width=7.5cm]{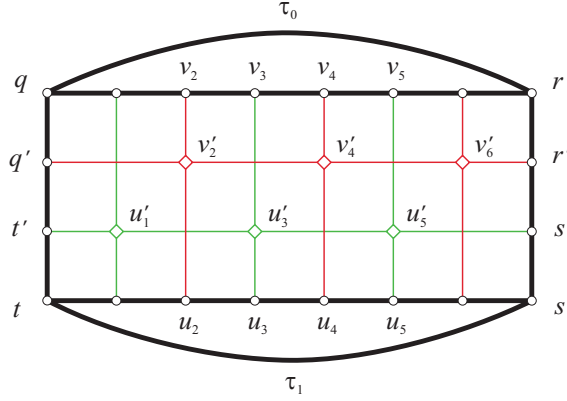}}
\caption{The graph~$H_n$ (for $n=6$).}
\label{fig:Hn}
\end {figure}

The $n^{\rm th}$ \emph{hamburger graph} $H_n$ is a special graph
with $3n+8$ vertices. Its thick edges form a cycle
$C = qv_1\dots v_n rr's'su_n\dots u_1 tt'q'q$ of length $2n+8$
together with two additional thick edges $\tau_0=qr$ and $\tau_1=st$.
See Figure \ref{fig:Hn}. In addition to these, $H_n$ has $n$ special
vertices $u_i'$ (for odd values of $i$) and $v_i'$ (for even values of $i$)
with rotation as shown in the figure. These vertices are of degree 4
and they lie on paths $r_1 = q'v_2'v_4'\dots v_m'r'$ (where $m=n$ if $n$ is
even and $m=n-1$ otherwise) and $r_2 = t'u_1'u_3'\dots u_l's'$ (where
$l=n$ if $n$ is odd and $m=n-1$ otherwise). These two paths will be referred
to as the \emph{rows} of $H_n$. Each $u_i'$ and each $v_i'$ is adjacent
to $u_i$ and $v_i$, and the 2-path $c_i = u_iu_i'v_i$
(or $c_i = u_iv_i'v_i$, depending on the parity of $i$)
is called a \emph{column} of $H_n$, $i=1,\dots,n$.

We claim that the hamburger graph $H_n$ has crossing
sequence $(n,n-1,0)$ whenever $n\ge 5$ (or $n=3$).
Although this does not
handle all possible sequences of the form $(a,b,0)$, as
discussed above, these are in some sense the most difficult and
counterintuitive cases. Indeed, a rather trivial modification of
these will be used to get all possible sequences.

Since it is quite easy to sketch proofs of $cr_0(H_n) = n$ and
$cr_2(H_n) = 0$, let us pause to do so here (rigorous arguments will
be given later).  The first of these equalities follows from the
observation that every row must meet every column in any planar
drawing in which thick edges are crossing-free.
The second equality follows from the observation that $H_n$ minus
the thick edges $\tau_0$, $\tau_1$ is a graph which can be embedded in the
sphere.  Using an extra handle for each of~$\tau_0$, $\tau_1$
gives an embedding of the whole graph in a surface of genus 2.  Of course, it
is possible to draw $H_n$ in the torus with only $n-1$ crossings by
starting with the drawing in the figure and then adding a handle to
remove one crossing.  In the third section we shall show that these
are indeed optimal drawings (for $n=3$ and $n\ge 5$).


\section{Gadgets}
\label{sec:gadgets}

The goal of this section is to establish Lemma~\ref{gadget_lem} which
permits us to use special graphs in our constructions.
Similar gadgets as used in our proof have been used previously,
cf., e.g., Pelsmajer et al. \cite{Pels} or Archdeacon et al. \cite{ABS}.
We include the constructions and proofs for reader's convenience. 

\subsection*{Thick edges}


For every $e \in E(G)$ choose positive integer $w(e)$ and replace
$e$ by a copy of~$L_{w(e)}$ whenever $w(e)>1$. Let $G'$ be the
resulting graph. We claim, that the crossing number of~$G'$ is the
same as the ``weighted crossing number'' of~$G$: each crossing of
edges~$e_1$, $e_2$ is counted $w(e_1) w(e_2)$-times. Obviously,
$cr(G')$ is at most that, as we can draw each $L_e$ sufficiently
close to where $e$~was drawn. Moreover, there is an optimal drawing
of this form (which proves the converse inequality): Given an
optimal embedding of~$G'$, consider the subgraph $L_e$ and from the
$w(e)$ paths of length~2 between its ``end-points'' pick the one,
that is crossed the least number of times. We can draw the whole
subgraph $L_e$ close to this path without increasing the number of
crossings.

This shows that we can ``simulate weighted crossing number'' by
crossing number of a modified graph. In particular, we can let
$w(e)=1$ for each ordinary edge and $w(e)>cr(G)$ for each thick edge $e$ of~$G$.
This proves Lemma~\ref{gadget_lem} for graphs with thick edges.

\begin {figure}
\centerline{\includegraphics[width=10.5cm]{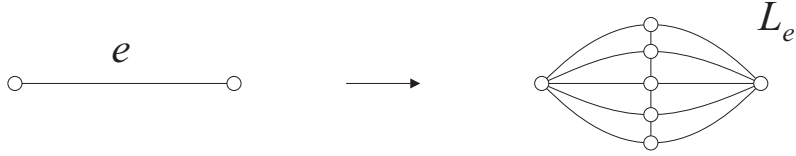}}
\caption{Putting weights on the edges (here $w(e)=5$).}
\label{fig:2a}
\end {figure}

\subsection*{Rigid vertices}

Suppose that we are considering drawings in surfaces of Euler genus
$\le g$; put $n=3g+2$. Let $G$ be a special graph with rigid vertices. We
replace each rigid vertex $v$ by a copy of $V_{n,\deg(v)}$. That is,
we add $n$ nested thick cycles of length $d=\deg(v)$ around~$v$ as
shown in Figure \ref{fig:2b} for $d=6$ and $n=5$. When doing this,
the cycles meet the edges incident with $v$ in the same order as
requested by the local rotation $\pi_v$ around $v$. If an edge
incident with $v$ is thick, then all edges in $G'$ arising from it
are thick too (as indicated in the figure for one of the edges).
Call the resulting graph $G'$.

\begin {figure}
\centerline{\includegraphics[width=10.5cm]{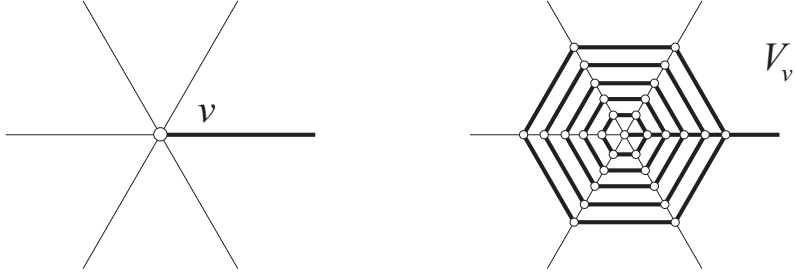}}
\caption{Controlling the prescribed local rotations.}
\label{fig:2b}
\end {figure}

We claim that the crossing number of~$G'$ (graph with thick edges but no rigid
vertices) is the same as that of~$G$.
Any drawing of~$G$ that respects the rotations at each rigid vertex
can be extended to a drawing of~$G'$ without any new crossing; in this
drawing all $n$ thick cycles in each $V_v$ are contractible and $v$ is
contained in the disc that any of them is bounding.
We will show, that there is an optimal drawing of~$G'$ of this
``canonical'' type.

Let us consider an optimal drawing (respecting thick edges) of $G'$ in $S$
(of genus $\le g$). Let $v$ be a rigid vertex of~$G$, and consider
the inner $n-1$ out of the $n$ thick cycles in~$V_v$. No edge of these
cycles is crossed; so by \cite[Proposition 4.2.6]{MT}, either one of these
cycles is contractible in $S$, or two of them are homotopic.

Suppose first, that one of the cycles, $Q$, is contractible.
Since $Q$ separates the graph into two connected components, either
the disk $D$ bounded by $Q$ or its exterior contains no vertex or edge of
$G'$ apart from some cycles and edges of $V_v$. Let us assume that this
is the interior of $D$. Now delete the drawing
of all thick cycles in $V_v$ except $Q$, and delete the drawing of all
$\deg(v)$ paths from $Q$ to $v$. Now think of $Q$ as the outermost cycle of
$V_v$ and draw the rest on $V_v$ inside $D$ without crossings.

Suppose next, that two of the cycles, $Q_1$ and~$Q_2$ are homotopic
(and that $Q_1$ is closer to $v$ in~$G'$). We cut $S$ along $Q_1$,
and patch the two holes with a disc. This simplifies the surface, so
if we can draw $G'$ on it without new crossings, we get a contradiction.
Such drawing of~$G'$ indeed exists, as we may delete the drawing of all of
$V_v$ that is ``inside'' $Q_1$ and draw it in one of the new discs.

By performing such a change to each rigid vertex, we obtain an optimum
drawing of $G'$ which is canonical. Consequently, it gives rise to a
legitimate drawing of the special graph $G$, and which is also optimal for
$G$. This shows that Lemma \ref{gadget_lem} holds also when there
are special vertices.


\section{Hamburgers}

The goal of this section is to prove Theorem~\ref{hamburger},
showing the existence of a graph with crossing sequence
$(a,b,0)$ for every $a>b>0$.  The hamburger graphs $H_n$
(defined in the introduction) have all of the key features of interest.
These are actually special graphs, but thanks to Lemma~\ref{gadget_lem} it
is enough to consider crossing sequences of special graphs.
Indeed, in the remainder of the paper we will omit the term `special'.

We have redrawn $H_n$ (for $n=5$) again in Figure~\ref{fig:Hn2} where we have given
names to numerous subgraphs of it. We have previously defined the rows
$r_1,r_2$ and columns $c_1,\dots,c_n$. For convenience we add rows
$r_0$ and $r_3$ and columns $c_0$ and $c_{n+1}$ (see Figure~\ref{fig:Hn2}).
The cycle $C$ (consisting of $c_0$, $r_0$, $c_{n+1}$, and $r_3$)   
has two trivial bridges (the thick edges $\tau_0$ and $\tau_1$)
and two other bridges. The first, denoted by $B_1$, consists of the row
$r_1$ together with all columns $c_i$ with $i$ even (and, of course, 
$1 \le i \le n$). The second one is denoted by $B_2$ and consists of the row $r_2$ and columns $c_i$ 
with $i$ odd (and, again, $1 \le i \le n$).

\begin {figure}
\centerline{\includegraphics[width=8.5cm]{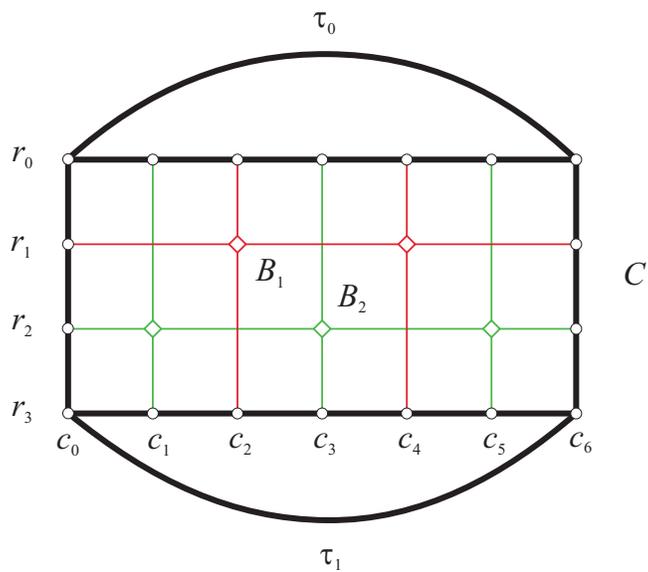}}
\caption{Main constituents of the graph $H_n$ (for $n=5$).}
\label{fig:Hn2}
\end {figure}

To get every possible crossing sequence $(a,b,0)$, we will also
require a slightly more general class of graphs.  For every $n,k \in \En$
with $n \ge 3$, we define the graph $H_{n,k}$, which is obtained from $H_n$
by adding $k$ duplicates of the second column $c_2$ as shown in
Figure~\ref{fig:Hnk} for the case of $n=4$ and $k=3$.
Note that $H_n \cong H_{n,0}$.

We shall denote by $\ss_g$ ($g\ge0$) the orientable surface of genus $g$.

\begin {figure}
\centerline{\includegraphics[width=7.5cm]{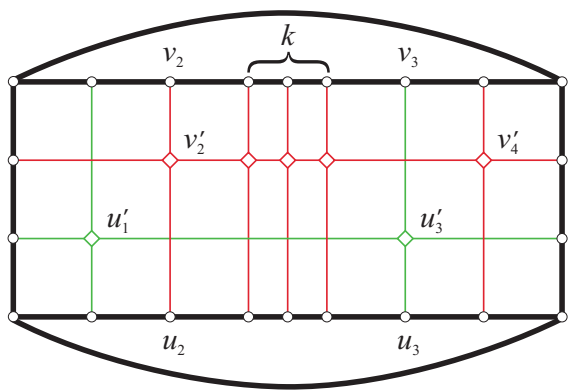}}
\caption{The graph~$H_{n,k}$ (for $n=4$ and $k=3$).}
\label{fig:Hnk}
\end {figure}

\begin{lemma} \label{ham_2tor}
 $cr_2(H_{n,k}) = 0$ for every $n,k \in \En$
with $n \ge 3$.
\end{lemma}

\begin{proof}
To draw $H_n$ in the double torus $\ss_2$, start by embedding
$H_n - \tau_0 - \tau_1$ in the sphere $\ss_0$.  Now, use one handle to route
the edge $\tau_0$, and another handle for $\tau_1$.
\end{proof}

\begin{lemma} \label{ham_plane}
$cr_0(H_{n,k}) = n+k$ for every $n,k \in \En$ with $n \ge 3$.
\end{lemma}

\begin{proof}
Consider a drawing of $H_{n,k}$ in the sphere. If this drawing has
finite crossing number, the cycle $C$ must be embedded as a simple
closed curve which separates the surface into two discs $D_1,D_2$ and is not
crossed by any edge. Moreover, both thick edges $\tau_0$ and $\tau_1$
are drawn in the same disc, say $D_2$. Now every column of $B_1$ crosses
the row $r_2$ and
every column of $B_2$ crosses the row $r_1$, so we have at least
$n+k$ crossings.  Since $H_{n,k}$ is drawn in $\ss_0$ with $n+k$ crossings
in Figure \ref{fig:Hnk}, we conclude that $cr_0(H_{n,k}) = n+k$ as required.
\end{proof}

Not surprisingly, the situation when drawing our graphs $H_n$ on the
torus is considerably more complicated to analyze.  By drawing $H_n$
in the plane with $n$ crossings and then using a handle to remove
one crossing, we see that $cr_1(H_n) \le n-1$ for all $n \ge 3$
(even $cr_1(H_{n,k}) \le n-1$ for all $n \ge 3$ and $k \ge 0$).
For $n \ge 5$, we shall prove that this is the best which can be
achieved.  For $n \le 4$, however, there is some exceptional
behavior (cf.\ Lemma~\ref{ham_lem}).

\begin{lemma}
For every optimal drawing of $H_n$ (in some surface), each column
$c_i$ $(1\le i\le n)$ is a simple curve.
\end{lemma}

\begin{proof}
It is easy to see that in every optimal drawing, every edge is represented
by a simple curve. Let us now consider a column $c_i=v_iv_i'u_i$ (or
similarly for $v_iu_i'u_i$) and suppose that the edges $e=v_iv_i'$ and
$f=u_iv_i'$ cross. Suppose that $e$ is represented by the simple curve
$\alpha(t)$, $0\le t\le 1$, where $\alpha(0)=v_i$ and $\alpha(1)=v_i'$.
Similarly, let $f$ be represented by the simple curve $\beta(t)$,
$0\le t\le 1$, where $\beta(0)=u_i$ and $\beta(1)=v_i'$.
Let $\alpha(t')=\beta(t')$ ($0<t'<1$) be where they cross.
Now let $\tilde\alpha(t)=\alpha(t)$ for $t\le t'$ and
$\tilde\alpha(t)=\beta(t)$ for $t\ge t'$. Change similarly $\beta$
to $\tilde\beta$. Then the crossing becomes a touching of the two curves,
which can be eliminated yielding a drawing with fewer crossings.
Observe that the local rotation at the special vertex $v_i'$ changes
from clockwise to anticlockwise but this is still consistent with the
requirement for this special vertex. Therefore the new drawing contradicts
the optimality of the original one.
\end{proof}

At several occasions in the proof we will use the following
well-known fact about closed curves on the torus.

\begin{lemma}[{\cite[Proposition 4.2.6]{MT}}] \label{l:fact}
Let $\varphi$, $\psi$ be two simple closed noncontractible
curves on the torus that are not freely homotopic. Then $\varphi$ and $\psi$ 
cross each other.
\end{lemma}

The following is well-known (cf., e.g., \cite{Zieschang}).

\begin{lemma} \label{l:twoclosedcurves}
Let $\varphi$, $\psi$ be two closed curves on some surface; assume $\psi$ is contractible.
The curves may intersect themselves and each other, but we assume that 
\begin{enumerate}
  \item the total number of intersections is finite, and
  \item each point of intersection is a crossing (the curves do not touch
    and there are no more than two arcs that run through the point).
\end{enumerate}

Then, the number of intersections of $\varphi$ with $\psi$ is even.
\end{lemma}

\paragraph{Proof (hint):} 
Let us transform $\psi$ continuously to a trivial curve. 
The number of intersections of $\varphi$ with $\psi$ stays the 
same, or changes by $2$ when we modify $\psi$ as in Figure~\ref{fig:Reidemeister}.
\bigskip


\begin {figure}
\centerline{\includegraphics[width=10.5cm]{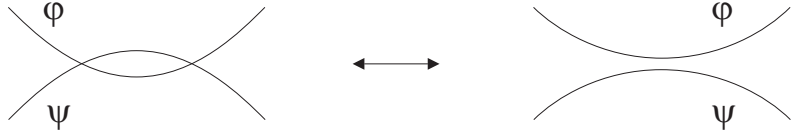}}
\caption{Illustration for the proof of Lemma~\ref{l:twoclosedcurves}.}
\label{fig:Reidemeister}
\end {figure}

It will be convenient for us to classify different types of drawings
of $H_n$ in the torus depending on the drawing of the thick subgraph
$C +\tau_0+\tau_1$. In Figure~\ref{fig:types} we have listed nine possible
embeddings of $C+\tau_0+\tau_1$ in $\ss_1$, where $\tau_0$ and $\tau_1$ are
drawn with dashed lines. We shall say that a drawing of
$H_n$ is of \emph{type} $A$, $B$, $C$, $C'$, $D$, $E$, $E'$, $E''$, or $E'''$ if the induced drawing of
$C + \tau_0 + \tau_1$ is as in the corresponding part of Figure~\ref{fig:types}.
Although there are other possible drawings of $C + \tau_0 + \tau_1$ in the torus, our next lemma shows that the only ones which extend to finite crossing number drawings of $H_n$ have one of these types.

\begin{lemma}\label{finite_hamburger}
Every drawing of $H_n$ for $n \ge 3$ on a torus $\s$ with 
crossing number less than~$n$ has type $A$, $B$, $C$, $C'$, $D$, $E$, $E'$, $E''$, or $E'''$.
\end{lemma}

\begin{proof} Let $\s'$ be the bordered surface obtained from $\s$ by cutting along the cycle $C$.
First suppose that $C$ is contractible.  Then $\s'$ is disconnected, with one component
a disc $D$, and the other component $\s''$ homeomorphic to $\ss_1$
minus a disc.  If both $B_1$ and $B_2$ are drawn in $D$, then we
have at least $n$ crossings (as in Lemma \ref{ham_plane}).  If only one
of $B_1$ or $B_2$, say $B_1$ is drawn in $D$, then $B_2$ and the
edges $\tau_0$ and $\tau_1$ are drawn in $\s''$ (else the crossing
number is infinite).  Consider the curves $\tau_0 \cup r_0$ and
$\tau_1 \cup r_3$ in $\s''$.  If either of these is contractible, then $B_2$
must cross it (yielding infinite crossing number).  Otherwise 
(using the Lemma~\ref{l:fact})
they must be freely homotopic noncontractible curves in $\s''$,
so $\tau_0 \cup c_0 \cup \tau_1 \cup c_{n+1}$ is a contractible curve.
Therefore $B_2$ must cross it, yielding again infinitely many
crossings.  Thus, we may assume that both $\tau_0$ and $\tau_1$ are
drawn in the disc $D$ and $B_1$ and $B_2$ are drawn in $\s''$ so our
drawing is of type~$A$.

Next suppose that $C$ is not contractible.  In this case, the
surface $\s'$ is a cylinder bounded by two copies of the cycle $C$.
If both $\tau_0$ and $\tau_1$ have all of their ends on the same
copy of $C$, we must have a drawing of type $B$, $C$, or $C'$. If one has
both ends on one copy of $C$, and the other has both ends on the
other copy of $C$, then there are infinitely many crossings, unless
the drawing is of type $D$. Finally, if one of $\tau_0$, $\tau_1$,
has its ends on distinct copies of $C$, then the
crossing number will be infinite unless the other one of $\tau_0$, $\tau_1$,
has both ends on
the same copy of $C$ giving us a drawing of type $E$, $E'$, $E''$, or $E'''$.
\end{proof}

If $G$ is a graph drawn on a surface and $A,B \subseteq G$, then we shall denote by
$Cr(A \mid B)$ the total number of crossings of an edge from~$A$ with an edge
from~$B$, where crossings of an edge $e\in E(A\cap B)$ with another edge
$f\in E(A\cap B)$ are counted only once. In particular,
the total number of crossings of graph~$G$ is equal to $Cr(G \mid G)$.

\begin {figure}
\centerline{\includegraphics[width=12cm]{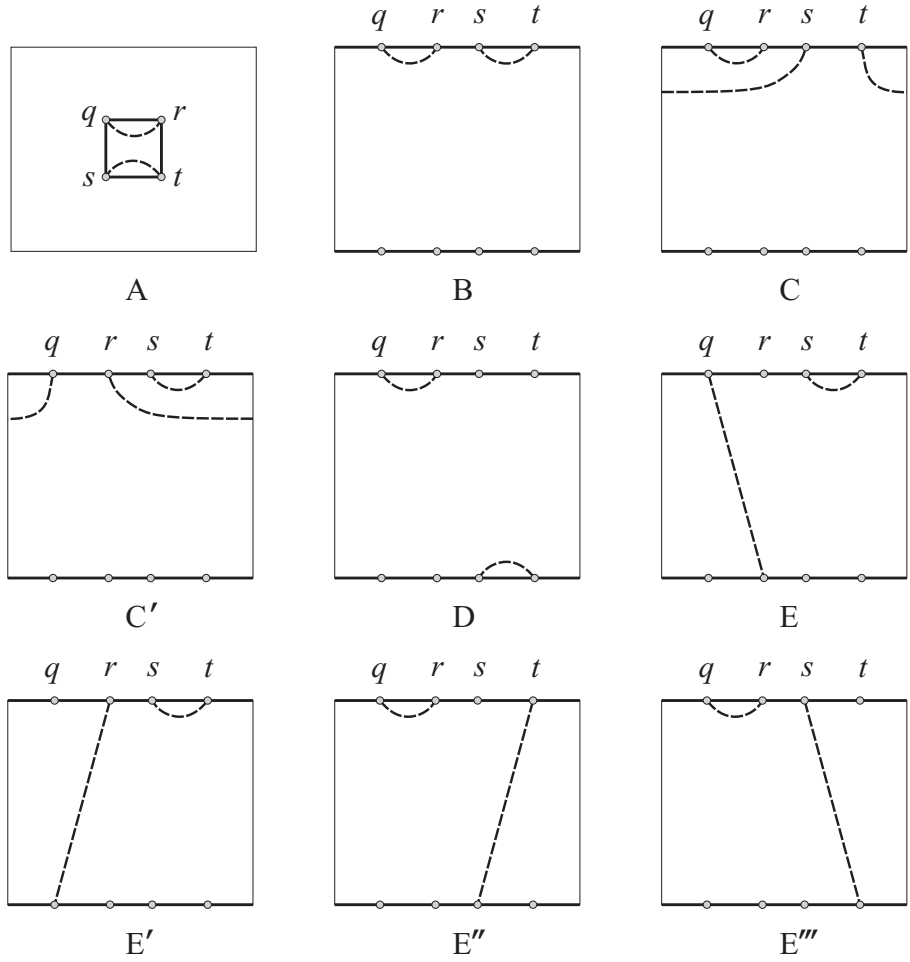}}
\caption{Nine special types of embedding of the thick subgraph
         $C + \tau_0 + \tau_1$ in the torus. In types $B$--$E'''$, the cycle
         $C$ is drawn on the top and bottom sides of the square.}
\label{fig:types}
\end {figure}

\begin{lemma} \label{ham_lem}
$cr_1(H_n) = n-1$ if\/ $n=3$ or\/ $n \ge 5$, while\/ $cr_1(H_4) =
2$. Furthermore, Figure \ref{fig:excH3}(a)--(c') shows the only
drawings of $H_3$ in the torus with two crossings and the added
property that $Cr(r_2 | G) = 0$. Figure~\ref{fig:excH4} displays the
unique drawing of $H_4$ in the torus with two crossings. 
\end{lemma}

\begin {figure}
\centerline{\includegraphics[width=12cm]{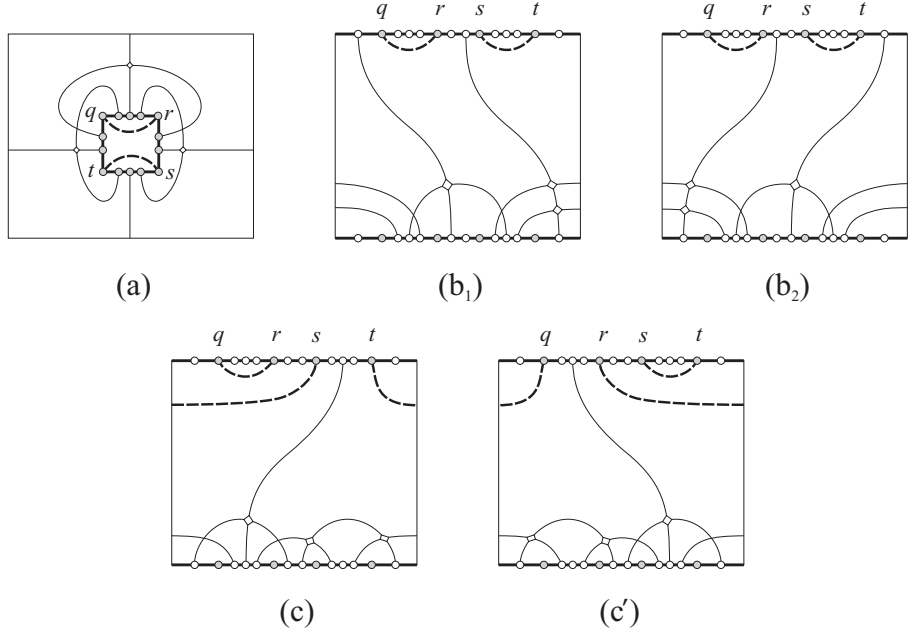}}
\caption{Exceptional drawings of $H_3$.}
\label{fig:excH3}
\end {figure}

\begin {figure}
\centerline{\includegraphics[width=3.6cm]{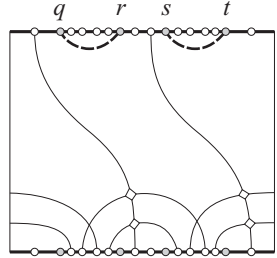}}
\caption{Exceptional type~$B$ drawing of $H_4$.}
\label{fig:excH4}
\end {figure}

\begin{proof}
We proceed by induction on~$n$.  Consider a drawing $\calD$ of~$H_n$ in a
surface $\s$ homeomorphic to the torus, such that $\calD$ yields minimum crossing
number.
We shall frequently use the inductive assumption for $n-1$ and~$n-2$, since by
deleting the edges of the column $c_1$, the column $c_n$, or two consecutive
columns $c_i$ and~$c_{i+1}$ we obtain a new graph which is a subdivision of
$H_{n-1}$ or $H_{n-2}$ (assuming $n \ge 3$). This technique
will be used throughout the proof. It is also worth noting that
after applying this operation to $\calD$, the drawing of the smaller
hamburger graph is of the same type as the drawing~$\calD$.

The cycle $C$ is not crossed in~$\calD$, so we may cut our
surface along this curve. This leaves us with a drawing of $H_n$ in
a closed bordered surface---which we shall denote $\s'$---where each
edge of $C$ appears twice on the boundary. We shall use $C^1$ and~$C^2$
to denote these copies.

Essential to our proof is an analysis of the homotopy behavior of
the rows and columns. To make this precise, let us now choose a
point $N$ in the interior of the row $r_0$, $S$ in the interior of
$r_3$, $W$ in the interior of $c_0$ and $E$ in the interior of
$c_{n+1}$. (Actually, for each of these points we have two copies:
$N^1$ and $N^2$, etc. But we will avoid distinguishing these if
there is no danger of confusion). For each column $c_i$ ($0 \le i
\le n+1$) let $c_i^+$ be a simple curve in~$\s'$ obtained by
extending $c_i$ along the appropriate copies of the rows $r_0$ and
$r_3$ so that it has ends $N$ and $S$.  Similarly, for each row
$r_i$ ($0 \le i \le 3$) let $r_i^+$ be a curve in $\s'$
obtained by extending $r_i$ along the appropriate copies of the
columns $c_0$ and $c_{n+1}$ so that it has ends $E$ and $W$. We
shall focus our attention on the homotopy types in $\s'$ of the
curves $c_i^+$ where $N$ and $S$ are the fixed end points (and
similarly $r_i^+$ where $E$ and $W$ are fixed): we say that $c_i^+$
and $c_j^+$ are \emph{homotopic} if $c_i^+$ may be continuously
deformed to $c_j^+$ in the surface $\s'$, while keeping their
endpoints fixed. Note that $c_i^+$ and $c_j^+$ can only be homotopic
if $c_i$ and $c_j$ are connecting the same copies of $N$ and
$S$---that is they attach on the same side of $C$ in the original
surface $\s$. Also note, that for $i=0$ or $i=n+1$ we actually have
two copies of~$c_i$, so we should be speaking of, e.g., $c_0^+{}^1$
and $c_0^+{}^2$. We will refrain from this distinction whenever
possible to keep the notation clearer---so when saying $c_0^+$ and
$c_1^+$ are homotopic we will actually mean that $c_1^+$ is
homotopic to $c_0^+{}^s$ for some~$s \in \{1, 2\}$.

\bigskip
We will use frequently the following fact that connects the
homotopy types of columns and their crossing behaviour with respect
to the rows (and vice versa). We will refer to this statement as to ``the Claim''.

\mytit{Claim:}{If $c_i^+$ and $c_{i+1}^+$ are homotopic
  ($1 \le i < n$),
  then $Cr( r_j \mid c_i\cup c_{i+1}) \ge 1$ for $j = 1, 2$.  Similarly, if $r_1^+$ and $r_2^+$
  are homotopic, then $Cr( r_1\cup r_2 \mid c_i ) \ge 1$ for every $1 \le i \le n$.}

\bigskip

To see this, let us observe that the closed curve obtained by following $c_i^+$ from $S$ to
$N$ and then $c_{i+1}^+$ from $N$ to $S$ is contractible, after deleting
part of its intersection with the cycle $C$, we get a contractible
curve $\psi$ that intersects itself only at finitely many points. 
The row $r_j$ must cross either $c_i^+$ or $c_{i+1}^+$ (depending on the parity)
in their common vertex (it cannot only touch it as their common vertex has
prescribed local rotation). We may extend $r_j^+$ into a closed curve
$\varphi$ by following closely along the cycle $C$.
This way we are adding two (or zero) intersections with~$\psi$. 
By Lemma~\ref{l:twoclosedcurves} curves $\varphi$ and $\psi$ have an
even number of intersection, thus $r_j$ must have another crossing with~$\psi$
and we are done.
The same argument holds when the rows and columns exchange their roles.

\mytit{Corollary:}{If $r_1^+$ and $r_2^+$ are homotopic, we are done,
  as there are at least $n$ intersections.}

\bigskip

In light of Lemma \ref{finite_hamburger} we may assume that our
drawing is of type $A$, $B$, $C$, $C'$, $D$, $E$, $E'$, $E''$, or $E'''$, and we now
split our argument into these nine cases.

\mytit{Case 1:}{Type $A$.}

Let us first suppose that $n \ge 4$. If there exists $1 \le i \le n$
so that $c_i^+$ is homotopic to $c_0^+$, then either $c_1$ crosses
$c_i$, or $c_1^+$ is homotopic to $c_0^+$. In the latter case, $c_1$
crosses $r_1$. So, in short, $Cr(c_1 \mid H_n) \ge 1$ and by
removing this column and applying induction, we deduce that there
are at least $n-1$ crossings in our drawing. Note here that the
resulting drawing of $H_{n-1}$ is still of type $A$, so it must have
at least $(n-1)-1$ crossings, even if $n=5$. Thus, we may assume that
$c_i^+$ is not homotopic to $c_0^+$ for any $1 \le i \le n$. By a
similar argument, $c_i^+$ is not homotopic to $c_{n+1}^+$.  If there
exist $i,j\in\{1,\dots,n\}$ with $c_i^+$ not homotopic to $c_j^+$,
then $c_i^+$ and $c_j^+$ cross (Lemma~\ref{l:fact}), and further,
$Cr(c_k \mid c_i\cup c_j) \ge 1$ for every $k \in \{1,\ldots,n\}$
with $k \neq i,j$. This implies that we have at least $n-1$
crossings, as desired.  The only other possibility is that $c_i^+$
and $c_j^+$ are homotopic for every $i,j \in \{1,\ldots,n\}$.  In
this case, it follows from the Claim (applied to $c_1^+$ and $c_2^+$,
$c_3^+$ and $c_4^+$, $\ldots$) that there are at least $n-1$
crossings.

Suppose now that $n = 3$. If $c_2^+$ is homotopic to $c_1^+$ or
$c_3^+$, then it follows from the Claim that each row has at least one
crossing, and we are done. Thus, we may assume that $c_2^+$ has
distinct homotopy type from that of $c_1^+$ and from that of
$c_3^+$.  If $c_2^+$ is homotopic to $c_0^+$, then $Cr( c_2 \mid
r_2) \ge 1$ and $Cr( c_2 \mid c_1 ) \ge 2$ (since $c_1^+$ is not
homotopic to $c_2^+$) giving us too many crossings.  Thus, $c_2^+$
is not homotopic to $c_0^+$, and by a similar argument, we find
that $c_2^+$ is not homotopic to $c_4^+$.  Now, either $c_1^+$ is
homotopic to $c_0^+$ (in which case $Cr(c_1 \mid r_1) \ge 1$) or
$c_1^+$ is not homotopic to $c_0^+$ (in which case $Cr( c_1 \mid
c_2 ) \ge 1$).  So, in short $Cr( c_1 \mid r_1 \cup c_2) \ge 1$.  By
a similar argument, $Cr( c_3 \mid r_1 \cup c_2) \ge 1$.  Since there
are at most two crossings, we must have $Cr( c_1 \cup c_3 \mid r_1
\cup c_2) = 2$ and this accounts for all of our crossings.  In
particular, this implies that $r_1$ and $r_2$ are simple curves.
 Since $Cr ( r_2 \mid G) = 0$, it follows that $r_2^+$ is not
homotopic to $r_0^+$ or $r_3^+$.  By the Claim, $r_1^+$ is not
homotopic to $r_2^+$, and this together with $Cr(r_1 \mid r_2)=0$
implies that $r_1^+$ is homotopic to $r_0^+$.  It follows from
this that $Cr( r_1 \mid c_i ) = 1$ for $i=1,3$ and this accounts for
all of the crossings.  Such a drawing is possible, but must be
equivalent with that in Figure~\ref{fig:excH3}(a).


\bigskip

In all the remaining cases, we have that $\s'$ is a cylinder, and in
our figures we have drawn $\s'$ with the boundary component $C^1$ on
the top and $C^2$ on the bottom.


\mytit{Case 2:}{Type $B$.}

Here all of the column curves $c_i^+$ have ends $N^2$ and $S^2$.
Recall that these are copies of $N$ and $S$ drawn at the ``bottom copy''
$C^2$ of $C$. Since all of these curves are simple, it follows that for
every $1 \le i \le n$, the curve $c_i^+$ is either homotopic to
  the simple curve $N^2$--$W^2$--$S^2$ in $C^2$ (we shall call this
  homotopy type~$\ell$),
  or to the simple curve $N^2$--$E^2$--$S^2$ in $C^2$ (homotopy type~$r$).
Let ${\bf a} = a_1 a_2 \ldots a_n$ be the word given by the
rule that $a_i$ is the homotopy type of $c_i^+$.
We now have the following simple crossing property.

\begin{enumerate}[label=P\arabic{*}., ref=P\arabic{*}]
\item \label{cro_1} If $a_i = r$ and $a_j = \ell$ where $1 \le i < j \le n$,
then $Cr(c_i \mid c_j) \ge 2$.
\end{enumerate}

If there exists an $i$ ($1 \le i \le n$) so that $Cr( c_i \mid H_n) \ge 4$,
then $n \ge 5$ (otherwise the drawing is not optimal), and by
removing $c_i$ and either $c_{i-1}$ or $c_{i+1}$ and applying the
theorem inductively to the resulting graph, we deduce that there are
at least $4 + cr_1(H_{n-2}) \ge n$ crossings in our drawing, a
contradiction. It follows from this and \ref{cro_1}, that either
${\bf a} = \ell^i r^{n-i}$ or ${\bf a} = \ell^i r \ell r^{n-i-2}$.
We now split into subcases depending on $n$.

Suppose first that $n = 3$.
If $a_1 = a_2 = \ell$ or $a_2 = a_3 = r$, then it follows from
the Claim that $Cr( r_j \mid c_1\cup c_2\cup c_3 ) \ge 1$ for $j=1,2$ and we are
finished.  Otherwise, ${\bf a}$ must be $\ell r \ell$ or $r \ell r$
and $Cr( c_2 \mid c_1\cup c_3) \ge 2$. These configurations are
possible, but require that our drawing is equivalent with the one
in Figure~\ref{fig:excH3}(b)---this comes from ${\bf a} = \ell r \ell$,
if ${\bf a} = r \ell r$ we get a mirror image.

Next we consider the case when $n = 4$ and ${\bf a} = \ell^i
r^{4-i}$. Applying the Claim for the columns $c_1,c_2$ and $c_3,c_4$
resolves the cases when ${\bf a}$ is one of $\ell^4$, $r^4$, or
$\ell^2 r^2$ (each gives at least four crossings---a contradiction).
Suppose that ${\bf a} = \ell^3 r$ (or, with the same argument, ${\bf
a} = \ell r^3$). It follows from the Claim that $Cr( c_1\cup c_2
\mid r_1\cup r_2 ) \ge 2$ and $Cr( c_2\cup c_3 \mid r_1\cup r_2) \ge
2$, so the only possibility for fewer than three crossings is that
our drawing has 2 crossings, both of which are between $c_2$ and the
rows $r_1$ and $r_2$.   But then $c_2$ does not cross $c_1$ or
$c_3$, so $c_2$ is separated from $c_0$ by $c_1^+ \cup c_3^+$, so
$Cr(r_1 \mid c_1 \cup c_3) > 0$, a contradiction.

Next suppose that $n=4$ and ${\bf a} = \ell^i r \ell r^{2-i}$.  If
${\bf a} = \ell^2 r \ell$, then it follows from \ref{cro_1} that $Cr
( c_3 \mid c_4 ) \ge 2$ and from the Claim that $Cr ( c_1\cup c_2
\mid r_1\cup r_2 ) \ge 2$, so we have at least four crossings---a
contradiction. Similarly ${\bf a} = r \ell r^2$ is impossible.  The
only remaining possibility is ${\bf a} = \ell r \ell r$.  In this
case, we have $Cr( c_2 \mid c_3) \ge 2$, so the only possibility is
that there are exactly two crossings, both between $c_2$ and $c_3$.
This case can be realized, but requires that our drawing is
equivalent to that of Figure~\ref{fig:excH4}.

Lastly, suppose that $n \ge 5$. Since ${\bf a} \in  \{\ell^i
r^{n-i}, \ell^i r \ell r^{n-i-2} \}$, either $a_1 = a_2 = \ell$ or
$a_{n-1} = a_n = r$.  As these arguments are similar, we shall
consider only the former case. Now, it follows from the Claim that
$Cr( c_1 \cup c_2 \mid r_1 \cup r_2 ) \ge 2$, so removing the first
two columns gives us a drawing of $H_{n-2}$ with at least two
crossings less than in our present drawing of $H_n$.  By applying
our theorem inductively to this new drawing, we find that the only
possibility for less than $n-1$ crossings is that $n=6$ and ${\bf a}
= \ell^3 r \ell r$.  In this case, we have $Cr( c_4 \mid c_5) \ge
2$, so we may eliminate two crossings by removing columns 4 and 5.
This leaves us with a drawing of a graph isomorphic to $H_4$ as
above with the pattern $\ell^3 r$.  It follows from our earlier
analysis, that this drawing has at least three crossings. This
completes the proof of this case.

\mytit{Case 3:}{Type $C$.}

Now each column curve has one end on the segment of $C^2$ between $q^2$ and
$r^2$. As above, every curve $c_i^+$ with both ends on $C^2$ must be homotopic
with either
  the simple curve $N^2$--$W^2$--$S^2$ in $C^2$ (denoted by $\ell$),
  or  with the simple curve $N^2$--$E^2$--$S^2$ in $C^2$ (homotopy type $r$).
Each row has both its ends on~$C^2$.

The homotopy types of the other column curves will be represented by
integers. Since $\s'$ is a cylinder, we may choose a continuous
deformation $\Psi$ of $\s'$ onto the circle $\ss^1$ with the
property that $C^1$ and $C^2$ map bijectively to $\ss^1$, and $N^2$
and $S^1$ map to the same point $x\in \ss^1$. Now, each curve
$c_i^+$ maps to a closed curve in $\ss^1$ from~$x$ to~$x$, and for
an integer $\alpha \in \Zet$, we say that $c_i^+$ has homotopy type
$\alpha$ if the corresponding curve in $\ss^1$ has
(counterclockwise) winding number $\alpha$.  It follows that $c_i^+$
and $c_j^+$ are homotopic if and only if they have the same homotopy
type. As before, we let ${\bf a} = a_1 a_2 \ldots a_n$ be the word
given by the rule that $a_i$ is the homotopy type of $c_i^+$.  We
now have the following crossing properties (for the
appropriate choice of ``clockwise'' direction), whenever $1 \le i < j \le n$:

\begin{enumerate}[label=P\arabic{*}., ref=P\arabic{*}]
\item \label{cro_2} $Cr( c_i \mid c_j ) \ge | a_i - a_j - 1 |$
                    if $a_i, a_j \in \Zet$.
\item \label{cro_3} $Cr( c_i \mid c_j) \ge 2$ if $a_i = r$ and $a_j = \ell$.
\item \label{cro_4} $Cr( c_i \mid c_j) \ge 1$ if
                      either $a_i = r$ and $a_j \in \Zet$
                      or $a_i \in \Zet$ and $a_j = \ell$.
\end{enumerate}

By choosing $\Psi$ appropriately, we may further assume that the
smallest integer $1 \le i \le n$ for which $a_i \in \Zet$ (if such $i$ exists)
satisfies $a_i = 0$. Again, we split into subcases depending on $n$.

Suppose first that $n=3$.  Note that every column of type $r$ or
$\ell$ separates the segment $q^2 t^2$ on $C^2$ from $r^2 s^2$.
Consequently, $Cr( r_1\cup r_2 \mid c_i ) \ge 1$ whenever $a_i \in
\{\ell,r\}$.  Next we shall consider the homotopy types of our rows.
If $r_1^+$ is not homotopic to $r_0^+$ or $r_3^+$, then $Cr(r_1
\mid r_1) \ge 1$ and further $Cr(r_1 \mid c_1 \cup c_3) \ge 2$ (as in
this case, $r_1$~separates $C^2$ from~$C^1$ and also segment
$q^2 r^2$ from $s^2 t^2$) which gives us too many crossings.
If $r_2^+$ is not homotopic to $r_0^+$ or $r_3^+$, then $Cr(r_2
\mid r_2) \ge 1$ and $Cr(r_2 \mid c_2) \ge 1$, and we have nothing
left to prove.  Thus, we may assume that $r_1^+$ (and also $r_2^+$) is
homotopic to one of $r_0^+$, $r_3^+$.  If $r_1^+$ and $r_2^+$ are
homotopic, then the Claim implies that there are at least three
crossings.  Hence, we may assume that $r_1^+$ is homotopic to
$r_0^+$ and $r_2^+$ to $r_3^+$ (the other possibility yields two
crossings and each row crossed).  It now follows from our
assumptions that $Cr(r_1 \mid c_i) \ge 1$ for $i=1,3$, so assuming
we have at most two crossings, our only crossings are between $r_1$
and $c_1$ and between $r_1$ and $c_3$.  If $a_i \in \Zet$ for $i \in
\{1,3\}$, then $c_i$ also crosses $r_2$ because of the requirements
concerning local rotations at the special vertices $u_1'$ and
$u_3'$.  It follows that there are at least three crossings unless
${\bf a} = \ell 0 \ell$, $\ell 0 r$, $r 0 \ell$, or $r 0 r$.  Each
of these, except $\ell 0 r$ gives at least three crossings by
(\ref{cro_4}). The remaining case is possible, but only as it
appears in Figure~\ref{fig:excH3}(c).

Suppose now that $n \ge 4$. If either $c_1$ or~$c_n$ is crossed,
then we delete it and use the induction hypothesis. If neither has a
crossing, then both $a_1$ and~$a_n$ are integers (otherwise
$Cr(c_1\cup c_n \mid r_1\cup r_2) \ge 1$ as above). It follows that
$a_1=0$, and $a_n=-1$ (otherwise $c_1$ and~$c_n$ cross). Now there
is no value for $a_2$ to avoid crossing with either $c_1$ or $c_n$.
Hence one of $c_1$, and $c_n$ is crossed, after all, and
we may use induction. This completes the proof of Case~3.

\mytit{Case 4:}{Type $C'$.}

This case is nearly identical to the previous one. We may define
the homotopy types for the columns to be $r$, $\ell$, or an integer,
exactly as before, so that the same homotopy properties are
satisfied. Then the analysis for $n \ge 4$ is identical, and the
only difference is the case when $n=3$.  As before, if $r_1^+$ is
not homotopic to $r_0^+$ or $r_3^+$, then $Cr(r_1 \mid r_1) \ge 1$
and $Cr(r_1 \mid c_1 \cup c_3) \ge 2$ giving us too many crossings.
Similarly, if $r_2^+$ is not homotopic to $r_0^+$ or $r_3^+$, then
$Cr(r_2 \mid r_2) \ge 1$ and $Cr(r_2 \mid c_2) \ge 1$ and there is
nothing left to prove.  Now, using the Claim, we deduce that $r_1^+$
is homotopic to $r_0^+$ and $r_2^+$ is homotopic to $r_3^+$. It
follows from this that $Cr( c_2 \mid r_2 ) \ge 1$.  If $a_2 \in {\mathbb Z}$ 
then, as the vertex $v_2'$ is rigid, it follows
that $Cr(c_2 \mid r_1) \ge 1$ and we have nothing left to prove.
Thus, we may assume that $a_2 \in \{\ell,r\}$.  If $a_i \in
\{\ell,r\}$ for $i=1$ or $i=3$, then $c_i$ crosses $r_1$ and we are
done.  Thus, we may assume that $a_1, a_3\in \Zet$.  It now follows
that $Cr(c_2 \mid c_1 \cup c_3) \ge 1$.  
This can be realized with exactly two crossings, but row $r_2$
must be crossed.



\mytit{Case 5:}{Type $D$.}

In this case, every column has one end on $r_0^2$ and one end on
$r_3^1$. We define the homotopy types of curves $c_i^+$ using
integers as in the previous case. Again, $c_i^+$ and $c_j^+$ are
homotopic if and only if they have the same homotopy type.  As
before, we let ${\bf a} = a_1 a_2 \ldots a_n$ be the word given by
the rule that $a_i$ is the homotopy type of $c_i^+$. And as before,
we have the following useful crossing property:

\begin{enumerate}[label=P\arabic{*}., ref=P\arabic{*}]
\item \label{cro_6}
   $Cr( c_i \mid c_j ) \ge |a_i - a_j - 1|$ if $1 \le i < j \le n $.
\end{enumerate}

Suppose first that $n \ge 4$.
If the first column $c_1$ does not cross any other columns, then
${\bf a} = 0(-1)^{n-1}$. Similarly, if the last column does not cross any other
columns, then ${\bf a} = 0^{n-1}(-1)$. Since these cases are mutually
exclusive for $n \ge 4$, either the first, or the last column contains a crossing.  
Then we may remove it and apply induction.

If $n = 3$, we proceed as follows.
Using~\ref{cro_6} (and the convention $a_1 = 0$) we get that
the number of crossings between the columns is at least
$
  |a_2 + 1| + |a_3 + 1| + |a_2 - a_3 - 1| \ge |a_2+1| + |a_2|
$
(using the triangle inequality). Symmetrically, we get another
lower bound for the number of crossings: $|a_3 + 1| + |a_3 + 2|$.
If any of these bounds is at least~3, we are done. It follows
that $a_2 \in \{0, -1\}$ and $a_3 \in \{-1, -2\}$.
Now, if there are two consecutive columns with the same homotopy type,
then each row will cross some of these columns, and we are done.
Consequently ${\bf a} = 0, -1, -2$.  It follows that
$Cr( c_1 \mid c_3 ) \ge 1$. If $c_2$ crossed either $c_1$ or $c_3$, then it
would have to cross the column twice---which would yield too many
crossings. Similarly, if
$Cr( c_1 \mid c_3) > 1$, then $Cr( c_1 \mid c_3) \ge 3$ and we would have too
many crossings.  It follows that the three columns $c_1$, $c_2$, $c_3$ are
drawn as in Figure~\ref{fig:typeD}. Now we have that $c_1$ and $c_3$ separate
$c_2$ from $c_0^1$, $c_0^2$, $c_{n+1}^1$, and $c_{n+1}^2$.  It
follows that $Cr( r_1 \mid c_1 \cup c_3 ) \ge 2$ giving us too many
crossings.

\begin {figure}
\centerline{\includegraphics[width=3.6cm]{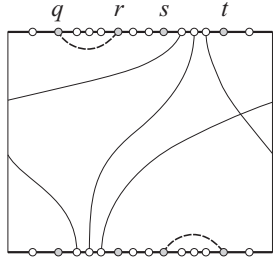}}
\caption{Part of a type~$D$ drawing of~$H_3$.}
\label{fig:typeD}
\end {figure}

\mytit{Case 6:}{Type $E$.}

In this case, every curve $c_i^+$ must have one end in $r_3^2$
and the other end in either $r_0^1$ or $r_0^2$.  In the first case,
we say that $c_i^+$ has homotopy type $0$ and in the second we
say it has type $\ell$.  It is immediate that any two such curves of the same
type are homotopic.  As usual, we let ${\bf a} = a_1a_2 \ldots a_n$ be
the word given by the rule that $a_i$ is the homotopy type of
$c_i^+$.  The following rule indicates some forced crossing behavior.

\begin{enumerate}[label=P\arabic{*}., ref=P\arabic{*}]
  \item \label{cro_7} $Cr( c_i \mid c_j ) \ge 1$ if $a_i = 0$
                      and $1 \le i < j \le n$.
\end{enumerate}

Let us first treat the case when $n \ge 4$.
If the last column $c_n$ contains at least one crossing, then we may
remove it and apply induction.  Otherwise,
(\ref{cro_7}) implies that ${\bf a} = \ell^n$ or ${\bf a} = \ell^{n-1}0$.  
It follows from the Claim that $Cr( c_1\cup c_2 \mid r_1\cup r_2 ) \ge 2$.
Thus, if $n \ge 5$, we may remove the first two columns and apply induction.
If $n = 4$ and ${\bf a} = \ell^4$, then the Claim
gives us at least four crossings---a contradiction with the minimality
of our drawing. It remains to check ${\bf a} = \ell^3 0$.  If there are fewer than three
crossings, then (again by applying the Claim twice) there are exactly two, and
both occur on $c_2$. However, in this case $Cr(r_1 \mid c_3) = 0$.
As $c_3$ separates $c_2$ from both $r^1s^1$ and $r^2s^2$ 
and $r_1$ has a common vertex with $c_2$, we get a contradiction.

\begin {figure}
\centerline{\includegraphics{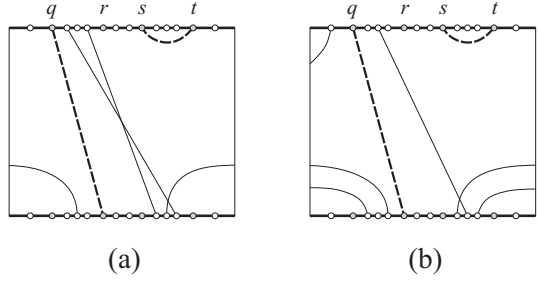}}
\caption{Towards type~$E$ drawings of~$H_3$.}
\label{fig:typeE}
\end {figure}

Finally, suppose that $n = 3$. If there are two consecutive columns
with the same homotopy type, then we are finished (by the Claim), so
we may assume ${\bf a} = 0\ell 0$ or ${\bf a} = \ell 0\ell$.  In the
former case, we have $Cr( c_1 \mid c_2\cup c_3 ) \ge 2$, so we may
assume that there are exactly two crossings, and the columns must be
drawn as in Figure~\ref{fig:typeE}(a). However, it is impossible to
complete this drawing to a drawing of $H_3$ with fewer than three
crossings. 

In the case ${\bf a} = \ell 0 \ell$ we have $Cr( c_2 \mid c_3) \ge 1$ 
(see Figure~\ref{fig:typeE}(b))
and the total number of crossings is at most two.  
If $r_2$ is crossed, then the drawing is not exceptional 
and we are done. There is a unique way to add $r_2$ to 
Figure~\ref{fig:typeE}(b) without creating any new crossing.
Then there is no way to add $r_1$ without crossing $r_2$.


\mytit{Case 7:}{Type $E'$.}

This case is very close to the previous one.  A similar analysis
reduces the problem to the case when $n=3$.  This case is actually
identical to the above: By reflecting both the torus pictured in
$E'$ and the standard drawing of $H_3$ (as in Figure 1) about a
vertical symmetry axis we find ourselves in this previous case.

\mytit{Case 8:}{Type $E''$.}

This case is somewhat similar to that of Type $E$.  We may define the
homotopy types for the columns $0$, $\ell$ exactly as before, so
that the crossing property (P1) from Type $E$ is satisfied.  
Then the analysis
for $n \ge 4$ is identical, and the only difference is the case when
$n=3$.  As before, if there are two consecutive columns with the
same homotopy type, we are finished. Thus we may assume that ${\bf
a} = 0 \ell 0$ or ${\bf a} = \ell 0 \ell$.  Then we get another
drawing of $H_3$ with two crossings, but again, 
in this case $r_1$ and $r_2$ cross each other.

\mytit{Case 9:}{Type $E'''$.}

This case is essentially the same as the previous one,
in the same way as type $E'$ was related to $E$.
This completes the proof of Lemma \ref{ham_lem}.
\end{proof}

Next we bootstrap to the following Lemma.

\begin{lemma}\label{lem x2}
The graph $H_{n,k}$ has crossing sequence $(n+k,n-1,0)$ for every $n
\ge 3$ and $k \ge 0$ with the exception of $n=4$ and $k=0$.
\end{lemma}

\begin{proof}
Lemmas \ref{ham_2tor} and \ref{ham_plane} show that 
$cr_0(H_{n,k}) = n+k$ and $cr_2(H_{n,k}) = 0$. We can draw $H_{n,k}$ 
in the torus with $n-1$ crossings by adding a handle to the drawing from 
Figure~\ref{fig:Hnk}. It remains to show that $cr_1(H_{n,k}) \ge n-1$
(for $n\ge 3$, unless $n=4$ and $k=0$).
Take a drawing of $H_{n,k}$ in the torus. By removing the $k$ extra
columns we obtain a drawing of $H_{n,0}$ in the torus, which 
(by Lemma~\ref{ham_lem}) has $\ge n-1$ crossings, unless $n=4$.
This completes the proof in all cases except when $n=4$.

If $n=4$, the same argument as above shows that $cr_1(H_{4,k}) \ge
cr_1(H_{4,1})$; we shall prove now that $cr_1(H_{4,1}) \ge 3$.
Suppose this is false, and consider a drawing of $H_{4,1}$ in the
torus with at most two crossings.  By removing the added column, we
obtain a drawing of $H_4$ in the torus with at most two crossings.
It follows from Lemma \ref{ham_lem} that this drawing is equivalent
to that in Figure \ref{fig:excH4}.  Since this drawing does not
extend to a drawing of $H_{4,1}$ with $\le 2$ crossings, this gives
us a contradiction.

Thus $H_{n,k}$ (for $(n,k) \ne (4,0)$), has crossing sequence
$(n+k,n-1,0)$ as claimed.
\end{proof}

Next we introduce one additional graph to get the crossing sequence
$(4,3,0)$.  We define the graph $H_3^+$ in the same way as $H_3$
except that we have three rows instead of two. See
Figure~\ref{fig:H3+}.

\begin {figure}
\centerline{\includegraphics[width=3.2cm]{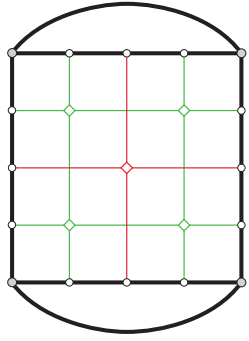}}
\caption{The special graph $H_3^+$.}
\label{fig:H3+}
\end {figure}

\begin{lemma} \label{lem_x3}
The graph $H_3^+$ has crossing sequence $(4,3,0)$
\end{lemma}

\begin{proof}
It follows from an argument as in Lemma \ref{ham_plane} that
$cr_0(H_3^+)=4$.  Since $H_3^+ - \tau_0 - \tau_1$ is planar, it
follows that $cr_2(H_3^+)=0$. It remains to show that $cr_1(H_3^+) =
3$.  Since $cr_1(H_3^+) \le 3$, we need only to show the reverse
inequality. Consider an optimal drawing of $H_3^+$ in the torus, and
suppose (for a contradiction) that it has fewer than three
crossings.  If the first row contains a crossing, then by removing
its edges, we obtain a drawing of a subdivision of $H_3$ in the
torus with at most one crossing---a contradiction. Thus, the first
row must not have a crossing, and by a similar argument, the third
row must not have a crossing.  Now, we again remove the first row.
This leaves us with a drawing of a subdivision of $H_3$ in the torus
with at most two crossings, and with the added property that one row
($r_2$ in this $H_3$)
has no crossings.  By Lemma \ref{ham_lem} this must be a drawing as
in Figure \ref{fig:excH3}.  A routine check of these drawings shows
that none of them can be extended to a drawing of $H_3^+$ with fewer
than 3 crossings.
\end{proof}

We require one added Lemma for some simple crossing sequences.

\begin{lemma}
For every $a > 1$ there is a graph with crossing sequence $(a,1,0)$.
\end{lemma}

\begin{proof}
Let $G_1$ be a copy of $K_5$, let $G_2$ be the graph obtained from a
copy of $K_5$ by replacing each edge, except for one of them, with
$a-1$ parallel edges joining the same pair of vertices. Let $G$ be
the disjoint union of $G_1$ and $G_2$.  It is immediate that
$cr_0(G) = a$, $cr_2(G) = 0$, and $cr_1(G) \ge 1$. A drawing of $G$
in $\ss_1$ with this crossing number is easy to obtain by embedding
$G_2$ in the torus, and then drawing $G_1$ disjoint from $G_2$ with
one crossing.  Thus, $G$ has crossing sequence $(a,1,0)$ as
required.
\end{proof}

\begin{proofof}{Theorem \ref{hamburger}}
Let $(a,b,0)$ be given with integers $a>b>0$ .  If $b=1$, then the
previous lemma shows that there is a graph with crossing sequence
$(a,b,0)$.  If $(a,b,0) = (4,3,0)$ then Lemma \ref{lem_x3} provides
such a graph.  Otherwise, Lemma \ref{lem x2} shows that the graph
$H_{b+1,a-b-1}$ has crossing sequence $(a,b,0)$.
\end{proofof}

\bibliographystyle{rs-amsplain}
\bibliography{cross-seq}

\end{document}